\documentclass{amsart}
\usepackage{amscd,amssymb}

\newtheorem{theorem}{Theorem}[section]
\newtheorem{proposition}[theorem]{Proposition}
\newtheorem{lemma}[theorem]{Lemma}
\newtheorem{corollary}[theorem]{Corollary}
\theoremstyle{definition}
\newtheorem{definition}[theorem]{Definition}

\theoremstyle{remark}

\numberwithin{equation}{section}

\DeclareMathOperator{\Mod}{Mod-}

\DeclareMathOperator{\Hom}{Hom}

\DeclareMathOperator{\Cyl}{Cyl}

\DeclareMathOperator{\ho}{Ho}

\newcommand{\Ch}[1]{\text{Ch} (#1)}
\newcommand{\cat}[1]{\mathcal{#1}}
\newcommand{\loc}[1]{\text{loc }\langle #1 \rangle}

\newcommand{\Z}{\mathbb{Z}}
\newcommand{\Ab}{\text{Ab}}
\newcommand{\SSet}{\mathbf{SSet}}
\newcommand{\Top}{\mathbf{Top}}

\newcommand{\colim}{\varprojlim}

\newcommand{\mathcolon}{\colon\,}

\newcommand{\uc}{\textup{:}}
\newcommand{\ulp}{\textup{(}}
\newcommand{\urp}{\textup{)}}
\newcommand{\usc}{\textup{;}}

\begin{document}

\title{The Eilenberg-Watts theorem in homotopical algebra}

\date{\today}

\author{Mark Hovey}
\address{Department of Mathematics \\ Wesleyan University
\\ Middletown, CT 06459}
\email{hovey@member.ams.org}


\begin{abstract}
The object of this paper is to prove that the standard categories in
which homotopy theory is done, such as topological spaces, simplicial
sets, chain complexes of abelian groups, and any of the various good
models for spectra, are all \textbf{homotopically self-contained}.
The left half of this statement essentially means that any functor
that looks like it could be a tensor product (or product, or smash
product) with a fixed object is in fact such a tensor product, up to
homotopy.  The right half says any functor that looks like it could be
Hom into a fixed object is so, up to homotopy.  More precisely,
suppose we have a closed symmetric monoidal category (resp. Quillen
model category) $\cat{M}$.  Then the functor $T_{N}\mathcolon
\cat{M}\xrightarrow{}\cat{M}$ that takes $M$ to $M\otimes N$ is an
$\cat{M}$-functor and a left adjoint.  The same is true if $N$ is an
$E\text{-}E'$-bimodule, where $E$ and $E'$ are monoids in $\cat{M}$,
and $T_{N}\mathcolon \Mod E\xrightarrow{}\Mod E'$ is defined by $T_{N}
(M)=M\otimes_{E}N$.  Define a closed symmetric monoidal category
(resp. model category) to be \textbf{left self-contained}
(resp. \textbf{homotopically left self-contained}) if every functor
$F\mathcolon \Mod E\xrightarrow{}\Mod E'$ that is an $\cat{M}$-functor
and a left adjoint (resp. and a left Quillen functor) is naturally
isomorphic (resp. naturally weakly equivalent) to $T_{N}$ for some
$N$.  The classical Eilenberg-Watts theorem in algebra then just says
that the category $\Ab$ of abelian groups is left self-contained, so
we are generalizing that theorem.
\end{abstract}

\maketitle

\section*{Introduction}

The object of this paper is to extend the Eilenberg-Watts
theorem~\cite{eilenberg, watts} to situations such as topological
spaces, chain complexes, or symmetric spectra, in which one is
interested in objects and functors not up to isomorphism, but up to
some notion of weak equivalence.  We first recall the standard
Eilenberg-Watts theorem.  

\begin{theorem}\label{thm-watts}
Let $R$ and $S$ be rings.  If $F\mathcolon \Mod R\xrightarrow{}\Mod S$
is additive and a left adjoint, then $FR$ is an $R\text{-}S$-bimodule
and there is a natural isomorphism
\[
X\otimes_{R} FR\xrightarrow{} FX.  
\]
\end{theorem}

Here, and throughout the paper, all modules are right modules unless
otherwise specified.  

This is not the usual formulation of the Eilenberg-Watts theorem,
which we now recall.

\begin{theorem}\label{thm-watts-old}
Let $R$ and $S$ be rings if $F\mathcolon \Mod R\xrightarrow{}\Mod S$
is additive, right exact, and preserves direct sums, then $FR$ is an
$R\text{-}S$-bimodule and there is a natural isomorphism
\[
X\otimes_{R} FR\xrightarrow{} FX.  
\]
\end{theorem}

These two theorems are equivalent, however, since a right exact
additive functor automatically preserves coequalizers, and then the
adjoint functor theorem implies that any colimit-preserving functor
between such nice categories is a left adjoint.  

We take the position that if $\cat{M}$ is a closed symmetric monoidal
category, $E$ and $E'$ are monoids in $\cat{M}$, and $F\mathcolon \Mod
E\xrightarrow{}\Mod E'$ is a left adjoint and an $\cat{M}$-functor
(the analogue of additive), then $F$ is a sort of generalized
$E\text{-}E'$-bimodule.  

We therefore make the following definition.  Given monoids $E$ and
$E'$ in a closed symmetric monoidal category $\cat{M}$ and an
$E\text{-}E'$-bimodule $N$, define $T_{N}\mathcolon \Mod
E\xrightarrow{}\Mod E'$ by $T_{N} (M)=M\otimes_{E}N$.  We will define
this more precisely in the first section below.  Similarly, for an
$E\otimes E'$-module $N$, define $S_{N}\mathcolon \Mod E\xrightarrow{}
(\Mod E')^{\textup{op}}$ by $S_{N} (M)=\Hom_{E} (M,N)$.  

\begin{definition}\label{defn-self}
Suppose $\cat{M}$ is a closed symmetric monoidal category.  We say
that $\cat{M}$ is \textbf{left self-contained} if, for every pair of
monoids $E, E'$ in $\cat{M}$ and every $\cat{M}$-functor $F\mathcolon
\Mod E\xrightarrow{}\Mod E'$ that is a left adjoint, there is an
$E\text{-}E'$-bimodule $N$ and a natural isomorphism $T_{N}\cong F$.
Similarly, $\cat{M}$ is \textbf{right self-contained} if, for every
pair of monoids $E,E'$ in $\cat{M}$ and every $\cat{M}$-functor
$F\mathcolon \Mod E\xrightarrow{} (\Mod E')^{\textup{op}}$ that is a
left adjoint, there is an $E\otimes E'$-module $N$ and a natural
isomorphism $F\cong S_{N}$.  We say that $\cat{M}$ is
\textbf{self-contained} if it is both left and right self-contained.
\end{definition}

The Eilenberg-Watts theorem stated above is then the assertion that
$\Ab$ is left self-contained.  In fact, Eilenberg~\cite{eilenberg}
also proved that $\Ab$ is right self-contained.  It would certainly be
interesting to go on from here to try to find out when a closed
symmetric monoidal category is self-contained, but the author is
primarily interested in homotopy theory, so we go in a different
direction.

Instead, we assume that we can also do homotopy theory in $\cat{M}$.
So we assume that $\cat{M}$ is a model category in the sense of
Quillen~\cite{quillen-htpy}.  The key idea in model categories is that
isomorphism is not the equivalence relation one cares about.  Instead,
there is a notion of weak equivalence.  Formally inverting the weak
equivalences gives the homotopy category $\ho \cat{M}$.  We would like
functors $F$ on $\cat{M}$ to induce functors $LF$ on $\ho \cat{M}$ in
a natural way, but this is generally true only for \textbf{left
Quillen} functors $F$ (for this and other model category terminology,
see~\cite{hirschhorn} or~\cite{hovey-model}).  

Thus we make the following definition. 

\begin{definition}\label{defn-self-htpy}
Suppose $\cat{M}$ is a closed symmetric monoidal model category.  We
say that $\cat{M}$ is \textbf{homotopically left self-contained} if,
for every pair of monoids $E,E'$ in $\cat{M}$ and every left Quillen
$\cat{M}$-functor $F\mathcolon \Mod E\xrightarrow{}\Mod E'$, there is
an $E\text{-}E'$-bimodule $N$ and a natural isomorphism
$LT_{N}\xrightarrow{}LF$ of functors on $\ho \Mod E$.  Similarly, $F$
is \textbf{homotopically right self-contained} if, for every pair of
monoids $E,E'$ in $\cat{M}$ and every left Quillen $\cat{M}$-functor
$F\mathcolon \Mod E\xrightarrow{} (\Mod E')^{\textup{op}}$, there is
an $E\otimes E'$-module $N$ and a natural isomorphism
$LF\xrightarrow{}LS_{N}$ of functors on $\ho \Mod E$.  And $\cat{M}$
is \textbf{homotopically self-contained} if it is both left and right
homotopically self-contained.
\end{definition}

For this definition to make precise sense, we need to assume enough
about $\cat{M}$ to be sure that $\Mod E$ inherits a model structure
from that of $\cat{M}$ for all monoids $E$, where the weak
equivalences (resp. fibrations) in $\ho \Mod E$ are maps that are weak
equivalences (resp. fibrations) when thought of as maps of $\cat{M}$.

The main result of this paper is then the following theorem, proved in
Section~\ref{sec-examples}.

\begin{theorem}\label{thm-main}
The following closed symmetric monoidal model categories are
homotopically self-contained. 
\begin{enumerate}
\item \ulp Compactly generated, weak Hausdorff\urp topological
spaces. 
\item Simplicial sets.  
\item Chain complexes of abelian groups. 
\item Symmetric spectra 
\item Orthogonal spectra 
\item $S$-modules.  
\end{enumerate}
\end{theorem}

The last three are all models of stable homotopy theory.  $S$-modules
were introduced in~\cite{elmendorf-kriz-mandell-may}, orthogonal
spectra in~\cite{mandell-may-schwede-shipley}, and symmetric spectra
in~\cite{hovey-shipley-smith}.  

To approach Theorem~\ref{thm-main}, we recall the proof of the
Eilenberg-Watts theorem.  The first step is to prove that, if
$F\mathcolon \Mod R\xrightarrow{}\Mod S$ is an additive functor, then
$FR$ is an $R\text{-}S$-bimodule, and there
is a natural transformation 
\[
M\otimes_{E} FE \xrightarrow{} FM.  
\]
This step is completely general. The following proposition is proved
as Proposition~\ref{prop-general-map}, where we will define any
unfamiliar terms.

\begin{proposition}\label{prop-enriched}
Suppose $\cat{M}$ is a closed symmetric monoidal category, $E$ and
$E'$ are monoids in $\cat{M}$, and 
\[
F\mathcolon \Mod E \xrightarrow{}\Mod E'
\]
is an $\cat{M}$-functor.  Then $FE$ is an $E\text{-}E'$-bimodule, and there
is a natural transformation 
\[
\tau \mathcolon X \otimes_{E} FE \xrightarrow{} FX.
\]
\end{proposition}

Given $\tau$, the proof of the usual Eilenberg-Watts theorem now proceeds
noting that $\tau$ is an isomorphism when $M=R$, both sides preserve
colimits, and $R$ generates $\Mod R$ under colimits.  There are
partial generalizations of this to the general case.  The nicest we
have is Theorem~\ref{thm-strict}, which asserts that $\tau$
is a natural isomorphism when $F$ is a strict $\cat{M}$-functor
(defined in Section~\ref{sec-strict}) and preserves coequalizers.

However, our main interest is in homotopical algebra, where it is a
mistake to ask whether a natural transformation is an isomorphism for
all $X$.  Instead, we ask whether it is a weak equivalence for all
cofibrant $X$; equivalently, we ask when the derived natural
transformation of $\tau $ is an isomorphism of functors on $\ho \Mod
E$.  For this to make sense, we need to assume enough about $\cat{M}$
so that we get model structures on $\Mod E$ and $\Mod E'$, and we need
to assume $F$ is a left Quillen functor as mentioned above.  We then
need a general theorem about when the derived natural transformation
$L\tau$ of a natural transformation $\tau$ of left Quillen functors is
an isomorphism. The author thinks that such a theorem should have been
proved before, but knows of no published reference.  The following
theorem is a combination of Theorem~\ref{thm-derived-cofibrant} and
Theorem~\ref{thm-derived-stable}.

\begin{theorem}\label{thm-Quillen}
Suppose $\cat{C}$ and $\cat{D}$ are model categories, $F,G\mathcolon
\cat{C}\xrightarrow{}\cat{D}$ are left Quillen functors, and $\tau
\mathcolon F\xrightarrow{}G$ is a natural transformation.  Suppose
that one of the two following conditions hold. 
\begin{enumerate}
\item $\cat{C}$ and $\cat{D}$ are stable, and there is a class
$\cat{G}$ of objects of $\ho \cat{C}$ such that the localizing
subcategory generated by $\cat{G}$ is all of $\ho \cat{C}$ and
$(L\tau) (X)$ is an isomorphism for all $X\in \cat{G}$.  
\item $\cat{C}$ is cofibrantly generated such that the domains of the
generating cofibrations of $\cat{C}$ are cofibrant, and $(L\tau) (X)$
is an isomorphism whenever $X$ is a domain or codomain of one of the
generating cofibrations.
\end{enumerate}
Then $L\tau$ is a natural isomorphism of functors on $\ho \cat{C}$.  
 \end{theorem}

We can then combine Theorem~\ref{thm-Quillen} with
Proposition~\ref{prop-enriched} to get versions of the Eilenberg-Watts
theorem in homotopical algebra.  Here are two of them, proved as
Theorem~\ref{thm-ultimate} and Theorem~\ref{thm-stable-monogenic}.
They use some terms we will define later.  

\begin{theorem}\label{thm-watts-homotopy}
Suppose $\cat{M}$ is a strongly cofibrantly generated, symmetric
monoidal model category.  Let $E$ and $E'$ be monoids in $\cat{M}$,
and $F\mathcolon \Mod E\xrightarrow{}\Mod E'$ be a left Quillen
$\cat{M}$-functor.  Suppose one of the hypotheses below holds.  
\begin{enumerate}
\item The domains of the generating cofibrations of $\cat{M}$ are
cofibrant, and the composite 
\[
A\otimes FQE\xrightarrow{}F (A\otimes QE) \xrightarrow{} F (A\otimes E)
\]
is a weak equivalence when $A$ is a domain or codomain of one of the
generating cofibrations of $\cat{M}$\usc \ or 
\item $\cat{M}$ is stable and monogenic with a cofibrant unit.  
\end{enumerate}
Then the natural transformation
\[
QX\otimes_{E} FQE \xrightarrow{} X\otimes_{E} FE \xrightarrow{} FQX
\]
is a natural isomorphism of functors on $\ho \Mod E$.  
\end{theorem}

Theorem~\ref{thm-main}, asserting that standard model categories are
homotopically self-contained, now follows.  Although we do not prove
Theorem~\ref{thm-main} in this way, a good way to think about it is
that $\tau$ is obviously an isomorphism for $X=E$, both the
functors $LT_{N}$ and $LF$ preserve suspensions and homotopy
colimits, and in the standard model categories, $E$ generates all of
$\ho \Mod E$ under the operations of homotopy colimits and
suspensions.  

We point out that Keller's work on DG-categories implies in particular
that chain complexes of abelian groups are homotopically
self-contained~\cite[Section~6.4]{keller}, in a stronger form than the
one we give.  We can recover Keller's full result by our methods using
the special features of the model structure on chain complexes of
abelian groups.

Note that the usual Eilenberg-Watts theorem is closely related to
Morita theory.  After all, if you have an equivalence of additive
categories $F\mathcolon \Mod R\xrightarrow{}\Mod S$, it satisfies the
hypotheses of the Eilenberg-Watts theorem, so must be given by
tensoring with a bimodule.  This is the beginning of Morita theory.
Similarly, our versions of the Eilenberg-Watts theorem are related to
the Morita theory of ring spectra due to Schwede and
Shipley~\cite{schwede-morita}.

The original motivation for this work was the important, yet
disturbing, paper of Christensen, Keller, and
Neeman~\cite{christensen-keller-neeman}, where they proved that not
all homology theories on the derived category $\cat{D} (R)$ of a
sufficiently complicated ordinary ring $R$ are representable.  This
followed work of Beligiannis~\cite{beligiannis}, who proved that not
every morphism between representable homology functors on $\cat{D}
(R)$ is representable.  Note that there is no problem with
representability of cohomology functors and morphisms between them.
This was rather a blow, and the author is not certain the field has
adequately adjusted to it yet.  This paper began by trying to find a
property of a homology functor $h$ that would ensure that it is
representable.  This question is discussed in the last section of the
paper.  

The author owes a debt of thanks to his former student Manny Lopez,
who first introduced him to the Eilenberg-Watts theorem.  He also
wishes to express his debt to Dan Christensen, Bernhard Keller, and
Amnon Neeman, both for writing the
paper~\cite{christensen-keller-neeman} and for their comments on an
early draft.

We should also mention Neeman's paper~\cite{neeman-dual}.
This paper considers representability for covariant exact functors on
triangulated categories that preserve \textbf{products}.  The author
thinks this needs further investigation, even in the abelian category
setting, although the situation is so different that such an
investigation would not reasonably fit into this paper.

\tableofcontents

\section{$\cat{M}$-functors}\label{sec-functors}

In this section, $\cat{M}$ is simply a closed symmetric monoidal
category with all finite colimits and limits.  We denote the monoidal
product by $\otimes$, the closed structure by $\Hom$, and the unit by
$S$.  The purpose of this section is to introduce enough background to
prove the following proposition.

\begin{proposition}\label{prop-general-map}
Suppose $\cat{M}$ is a closed symmetric monoidal category with finite
colimits and limits, and $E$ and $E'$ are monoids in $\cat{M}$.  If
$F\mathcolon \Mod E\xrightarrow{}\Mod E'$ is an $\cat{M}$-functor,
then $FE$ is an $E\text{-}E'$-bimodule, and there is a natural transformation
\[
\tau \mathcolon X\otimes_{E} FE \xrightarrow{} FX 
\]
of $E'$-modules that is an isomorphism when $X=E$.  Similarly, if $F$
is a contravariant $\cat{M}$-functor, then $FE$ is an
$E\otimes E'$-module, and there is a natural transformation
\[
\tau \mathcolon FX \xrightarrow{} \Hom_{E} (X, FE)
\]
of $E'$-modules that is an isomorphism when $X=E$.
\end{proposition}

For this proposition to make sense, we recall that a \textbf{monoid}
$E$ in a monoidal category $\cat{M}$ is the object $E$ equipped with a
unit map $S\xrightarrow{}E$ and a multiplication map $E\otimes
E\xrightarrow{}E$ that is unital and associative.  In this case, a
(right) module $X$ over $E$ is an object $X$ equipped with an action
map $X\otimes E\xrightarrow{}E$ that is associative and unital.  The
category $\Mod E$ is the category of such right modules and module
maps, which are of course maps in $\cat{M}$ compatible with the
actions.

The category $\Mod E$ then becomes a closed (left) module category
over $\cat{M}$, which means it is tensored, cotensored, and enriched
over $\cat{M}$.  In more detail, the tensor of $K\in \cat{M}$ and
$X\in \Mod E$, is the object $K\otimes X\in \Mod E$, where $E$ acts
only on $X$.  The cotensor of $K\in \cat{M}$ and $X\in \Mod E$ is
$\Hom (K,X)\in \Mod E$.  The action of $E$ is defined by 
\[
\Hom (K,X)\otimes E \cong \Hom (K,X) \otimes \Hom (S,E)
\xrightarrow{\otimes} \Hom (K, X\otimes E) \xrightarrow{}\Hom (K,X), 
\]
where the last map is induced by the action on $X$ (recall that $S$ is
the unit).  The enrichment $\Hom_{E} (X,Y)\in \cat{M}$ for $X,Y\in
\Mod E$ is defined as the equalizer of the evident two maps 
\[
\Hom (X,Y) \xrightarrow{} \Hom (X\otimes E,Y),
\]
one of which uses the action on $X$ and the other of which uses the
action on $Y$.  We have the usual adjunction isomorphisms 
\[
\Mod E (K\otimes X,Y) \cong \Mod E (X, \Hom (K,Y)) \cong \cat{M} (K,
\Hom_{E} (X,Y)).  
\]

Now, if $X$ is a right $E$-module and $Y$ is a left $E$-module, we can
form the tensor product $X\otimes_{E}Y\in \cat{M}$ as the coequalizer
of the two maps 
\[
X\otimes E\otimes Y\xrightarrow{}X\otimes Y.
\]
We can also define $E\text{-}E'$-bimodules in the usual way.  In this case,
the action map 
\[
X\otimes E' \xrightarrow{}X
\]
must be a map of left $E$-modules, which is equivalent to the left
$E$-action being a map of right $E'$-modules.  If $Y$ is an
$E\text{-}E'$-bimodule (with the evident definition), then $X\otimes_{E}Y$ is
in $\Mod E'$, because the action map
\[
X\otimes Y \otimes E'\xrightarrow{}X\otimes Y
\]
descends through the coequalizer.  Similarly, if $Y$ is both an
$E$-module and an $E'$-module in compatible fashion, which is
equivalent to $Y$ being an $E\otimes E'$-module, then $\Hom_{E} (X,Y)$
is naturally an $E'$-module.  

Finally, if $F\mathcolon \Mod E\xrightarrow{}\Mod E'$ is a functor, we
say that $F$ is an \textbf{$\cat{M}$-functor} if $F$ is compatible
with (one of, and hence all of) the tensor, cotensor, and enrichment
over $\cat{M}$.  That is, if $F$ is an $\cat{M}$-functor, then there
are natural maps 
\begin{gather*}
K\otimes FX\xrightarrow{} F (K\otimes X), \ \ \ \ F (\Hom
(K,X))\xrightarrow{} \Hom (K,FX) \\
\text{ and } \Hom_{E} (X,Y) \xrightarrow{} \Hom_{E'} (FX,FY)
\end{gather*}
satisfying all the properties one would expect.  If $F$ is
contravariant, then the roles of the tensor and cotensor are reversed,
so $F$ is an $\cat{M}$-functor when there are natural maps
\begin{gather*}
F (K\otimes X)\xrightarrow{}\Hom (K,FX), \ \ \ \ K\otimes
FX\xrightarrow{}F (\Hom (K,X)) \\
\text{ and } \Hom_{E} (X,Y) \xrightarrow{}\Hom_{E'} (FY,FX)
\end{gather*}
satisfying all the properties one would expect.  As an example of how
to get these maps from the enrichment 
\[
\Hom_{E} (X,Y)\xrightarrow{}\Hom_{E'} (FY,FX), 
\]
the adjoint to the identity of $K\otimes X$ is a map 
\[
K\xrightarrow{} \Hom_{E} (X,K\otimes X).  
\]
Composing this with the map 
\[
\Hom_{E} (X, K\otimes X) \xrightarrow{} \Hom_{E'} (F (K\otimes X), FX)
\]
and taking the adjoint gives us the map 
\[
F (K\otimes X)\xrightarrow{} \Hom (K,FX). 
\]

With this background in hand, we can now prove
Proposition~\ref{prop-general-map}.  

\begin{proof}
We begin with the covariant case.  For any monoid $E$, we have an
isomorphism of monoids
\[
E \cong \Hom (S,E) \cong \Hom_{E} (E,E). 
\]
If we had elements, we would say this takes $x\in E$ to the map
$E\xrightarrow{}E$ that is left multiplication by $x$.  Using the fact
that $F$ is an $\cat{M}$-functor, we get a map 
\[
\Hom_{E} (E,E) \xrightarrow{} \Hom_{E'} (FE,FE)
\]
of monoids.  The adjoint 
\[
E\otimes FE \xrightarrow{} FE
\]
in $\Mod E'$ to the composite map 
\[
E\xrightarrow{}\Hom_{E'} (FE,FE)
\]
maps $FE$ an $E\text{-}E'$-bimodule, though there are many details for the
conscientious reader to check.  

Similarly, the map 
\[
X\xrightarrow{} \Hom_{E} (E,X) \xrightarrow{} \Hom_{E'} (FE, FX)
\]
has adjoint the desired natural transformation 
\[
\tau \mathcolon X\otimes_{E} FE \xrightarrow{} FX
\]
of $E'$-modules.  There are even more details to check here.  In
particular, \emph{a priori} the adjoint to 
\[
\phi_{X}\mathcolon X\xrightarrow{} \Hom_{E'} (FE,FX)
\]
is just a map 
\[
X\otimes FE\xrightarrow{}FX
\]
of right $E'$-modules.  However, $\phi_{X}$ is in fact a map of right
$E$-modules, using the left $E$-module structure on $FE$ to make the
target of $\phi_{X}$ a right $E$-module.  This means that the adjoint
descends through the relevant coequalizer diagram to give the desired
map 
\[
\tau \mathcolon X\otimes_{E} FE \xrightarrow{} FX.  
\]

We now assume that $F$ is contravariant.  We get the right $E$-module
structure on the $E'$-module $FE$ via the adjoint to the map of
monoids
\[
E^{\textup{op}} \xrightarrow{}\Hom_{E} (E,E)^{\textup{op}}
\xrightarrow{}\Hom_{E'} (FE,FE).
\]
Here, if $X$ is a monoid, $X^{\textup{op}}$ is the monoid with the
reversed multiplication.  Similarly, we have the map
\[
X\cong \Hom_{E} (E,X) \xrightarrow{}\Hom_{E'} (FX, FE),
\]
which is in fact a map of right $E$-modules.  This has adjoint a map 
\[
FX\xrightarrow{}\Hom (X, FE),
\]
which in fact factors through $\Hom_{E} (X,FE)$, giving us the desired
natural transformation.  
\end{proof}

\section{A strict analog of the Eilenberg-Watts theorem}\label{sec-strict}

We are most interested in analogues of the Eilenberg-Watts theorem
that combine homotopy theory and algebra.  But we can also prove a
purely categorical version of the Eilenberg-Watts theorem, and we do
so in this section. 

We begin with a lemma about the structure of modules over a monoid $E$
in a symmetric monoidal category $\cat{M}$.  Define an $E$-module to
be \textbf{extended} if it is isomorphic to one of the form $X\otimes
E$, for some $X\in \cat{M}$.

\begin{lemma}\label{lem-E-module}
Suppose $\cat{M}$ is a symmetric monoidal category with all
coequalizers, and $E$ is a monoid in $\cat{M}$.  Then any $E$-module
is a coequalizer in $\Mod E$ of a parallel pair of morphisms
$P_{1}\rightrightarrows P_{0}$ between extended $E$-modules.
\end{lemma}

\begin{proof}
Given an $E$-module $X$, the action map $p\mathcolon X\otimes
E\xrightarrow{}X$ is a map of $E$-modules when we give $X\otimes E$
the free $E$-module structure.  This map is an $\cat{M}$-split
epimorphism by the map $i\mathcolon X\xrightarrow{}X\otimes E$ induced
by the unit of $E$.  Hence $X$ is the coequalizer in $\cat{M}$ of
$ip\mathcolon X\otimes E\xrightarrow{}X\otimes E$ and the identity of
$X\otimes E$.  It follows that $X$ is the coequalizer in $\Mod E$ of
\[
1\otimes \mu \mathcolon X\otimes E\otimes E\xrightarrow{}X\otimes E,
\]
where $\mu$ denotes the multiplication in $E$, and the composite
\[
X\otimes E\otimes E\xrightarrow{p\otimes 1} X\otimes E
\xrightarrow{i\otimes 1} X\otimes E\otimes E \xrightarrow{1\otimes
\mu} X\otimes E .
\]
This requires a little argument, using the fact that colimits in $\Mod
E$ can be calculated in $\cat{M}$.  
\end{proof}

We then get a general version of the Eilenberg-Watts theorem, whose
proof is straightforward enough, given Lemma~\ref{lem-E-module}, to
leave to the reader.

\begin{theorem}\label{thm-iso}
Let $\cat{M}$ be a closed symmetric monoidal category with all
finite colimits and limits, $E, E'$ be monoids in $\cat{M}$, and
$F\mathcolon \Mod E\xrightarrow{}\Mod E'$ be an $\cat{M}$-functor that
preserves coequalizers.  Suppose the natural map
\[
\tau_{X}\mathcolon X\otimes_{E} FE \xrightarrow{} FX
\]
is an isomorphism for all extended $E$-modules $X$.  Then it is a
natural isomorphism for all $E$-modules $X$.  Similarly, if
$F\mathcolon \Mod E\xrightarrow{}\Mod E'$ is a contravariant
$\cat{M}$-functor that takes coequalizers to equalizers, and the
natural map
\[
\tau_{X}\mathcolon FX\xrightarrow{} \Hom_{E} (X, FE)
\]
is an isomorphism for all extended $E$-modules $X$, then it is an
isomorphism for all $E$-modules $X$.  
\end{theorem}

Note that Theorem~\ref{thm-iso} does not imply the usual
Eilenberg-Watts theorem directly, though the method of proof does do
so.  Indeed, when $\cat{M}$ is abelian groups, every $E$-module is a
coequalizer of a map of free $E$-modules (not just extended ones), so
it suffices to know that $F$ is right exact (which is equivalent to
preserving coequalizers in this case) and $\tau_{X}$ is an isomorphism
on free $E$-modules.  For this, we need $F$ to preserve direct sums,
and then we recover the usual Eilenberg-Watts theorem.  

There is a special case in which $\tau_{X}$ is automatically an
isomorphism on extended $E$-modules.  An $\cat{M}$-functor
$F\mathcolon \Mod E\xrightarrow{}\Mod E'$ is called a \textbf{strict
$\cat{M}$-functor} if the structure map
\[
K\otimes FX \xrightarrow{}F (K\otimes X)
\]
is an isomorphism for all $K\in \cat{M}$ and $X\in \Mod E$.  In the
contravariant case, $F$ is a strict $\cat{M}$-functor if the structure
map 
\[
F (K\otimes X) \xrightarrow{}\Hom (K,FX)
\]
is an isomorphism for all $K\in \cat{M}$ and $X\in \Mod E$.  

\begin{theorem}\label{thm-strict}
Suppose $\cat{M}$ is a closed symmetric monoidal category with all
finite colimits and limits, $E, E'$ are monoids in $\cat{M}$, and
$F\mathcolon \Mod E\xrightarrow{}\Mod E'$ is a strict
$\cat{M}$-functor that preserves coequalizers.  Then the natural
map
\[
\tau_{X}\mathcolon X\otimes_{E} FE \xrightarrow{} FX
\]
is an isomorphism for all $X$.  Similarly, if $F\mathcolon \Mod
E\xrightarrow{}\Mod E'$ is a contravariant strict $\cat{M}$-functor
that takes coequalizers to equalizers, then the natural map
\[
FX\xrightarrow{} \Hom_{E} (X, FE)
\]
is an isomorphism for all $X$.  
\end{theorem}

\begin{proof}
It suffices to check that $\tau_{X}$ is an isomorphism on all extended
$E$-modules.  But if $X\in \cat{M}$, 
\[
F (X\otimes E)\cong X\otimes FE \cong (X\otimes E)\otimes_{E} FE
\]
since $F$ is a strict $\cat{M}$-functor.  The contravariant case is
similar.  
\end{proof}

\section{Monoidal model categories}\label{sec-model}

Since we are most interested in versions of the Eilenberg-Watts
theorem ``up to homotopy'', we now need to introduce the relevant
homotopical algebra.  For this, we will of course need to assume
knowledge of model categories, for which see~\cite[Part~2]{hirschhorn}
or~\cite{hovey-model}.

Our base category $\cat{M}$ will be both closed symmetric monoidal and
have a model structure.  Obviously, we will need to assume some
compatibility between the model structure and the monoidal structure
on $\cat{M}$.  This is well understood in the theory of model
categories, and we adopt the following definition,

\begin{definition}\label{defn-monoidal-model}
\begin{enumerate}
\item Suppose $\cat{M}$ is a monoidal category, and $f\mathcolon
A\xrightarrow{}B$ and $g\mathcolon C\xrightarrow{}D$ are maps in
$\cat{M}$.  The \textbf{pushout product} of $f$ and $g$, written
$f\Box g$, is the map 
\[
(A\otimes D) \amalg_{A\otimes C} (B\otimes C) \xrightarrow{} B\otimes D
\]
from the pushout of $A\otimes D$ and $B\otimes C$ over $A\otimes C$,
to $B\otimes D$.  
\item Now suppose $\cat{M}$ is a symmetric monoidal category equipped
with a model structure.  We say that $\cat{M}$ is a \textbf{symmetric
monoidal model category} if the following conditions hold\uc 
\begin{enumerate}
\item If $f$ and $g$ are cofibrations, then $f\Box g$ is a
cofibration. 
\item If $f$ is a cofibration and $g$ is a trivial cofibration, then
$f\Box g$ is a trivial cofibration.  
\end{enumerate}
\end{enumerate}
\end{definition}

This definition is different
from~\cite[Definition~4.2.6]{hovey-model}, where a unit condition was
added to ensure that the homotopy category of a symmetric monoidal
model category has a unit.  It automatically holds when the unit is
cofibrant.  This unit condition is not needed for some
of our versions of the Eilenberg-Watts theorem, and when it is needed,
it is much easier to assume the unit is cofibrant, so we omit it.  

But we also need the category $\Mod E$ of modules over a
monoid $E$ in $\cat{M}$ to be a model category, in a way that is
compatible with $\cat{M}$.  We therefore make the following
definition.  
 
\begin{definition}\label{defn-strongly-cofibrantly-generated}
Suppose $\cat{M}$ is a closed symmetric monoidal model category.  We
say that $\cat{M}$ is \textbf{strongly cofibrantly generated} if there
are sets $I$ of cofibrations in $\cat{M}$ and $J$ of trivial
cofibrations in $\cat{M}$ such that, for every monoid $E$ in
$\cat{M}$, the sets $I\otimes E$ and $J\otimes E$ cofibrantly generate
a model structure on $\Mod E$ where the weak equivalences are maps of
$E$-modules that are weak equivalences in $\cat{M}$.  
\end{definition}

This definition implies that the maps of $I\otimes E$ are small with
respect to $(I\otimes E)$-cell, and similarly for $J\otimes E$, as
this is part of the definition of cofibrantly generated.  Note also
that a map $p$ is a fibration (resp. trivial fibration) in $\Mod E$ if
and only if it has the right lifting property with respect to the maps
of $J\otimes E$ (resp. $I\otimes E$), which is equivalent to $p$ being
a fibration (resp.  trivial fibration) in $\cat{M}$.  This definition
also implies that the functor that takes $X\in \cat{M}$ to $X\otimes
E\in \Mod E$ is a left Quillen functor.

We note that pretty much every closed symmetric monoidal model
category that is commonly studied is strongly cofibrantly generated.
The most useful theorem along these lines is the following, which is a
paraphrase of~\cite[Theorem~4.1]{schwede-shipley-monoids}.

\begin{theorem}[Schwede-Shipley]\label{thm-strongly}
Suppose $\cat{M}$ is a cofibrantly generated, closed symmetric
monoidal model category.  If $\cat{M}$ satisfies the monoid axiom
of~\cite[Definition~3.3]{schwede-shipley-monoids} and every object of
$\cat{M}$ is small, then $\cat{M}$ is strongly cofibrantly generated.  
\end{theorem}

In fact, one does not need every object of $\cat{M}$ to be small.  Not
every topological space is small, but the category of (compactly
generated weak Hausdorff) topological spaces is still strongly
cofibrantly generated.  

One advantage of the strongly cofibrantly generated hypothesis is the
following proposition.  

\begin{proposition}\label{prop-strongly}
Suppose $\cat{M}$ is a strongly cofibrantly generated, closed
symmetric monoidal model category, and $E$ is a monoid in
$\cat{M}$. If $f$ is a cofibration in $\cat{M}$ and $g$ is a
cofibration in $\Mod E$, then the pushout product $f\Box g$ is a
cofibration in $\Mod E$, which is a trivial cofibration if either $f$
or $g$ is so.  In particular, if $A$ is cofibrant in $\Mod E$, and $f$
is a cofibration in $\cat{M}$, then $f\otimes A$ is a cofibration in
$\Mod E$, and is a trivial cofibration if $f$ is so.  
\end{proposition}

This proposition means that $\Mod E$ is an $\cat{M}$-model category,
in the language of~\cite[Definition~4.2.18]{hovey-model}, except that
we again have omitted a unit condition.  

\begin{proof}
The last sentence follows from the rest of the proposition by taking
$g$ to be the map $0\xrightarrow{}A$.  It suffices to check the
statement about $f\Box g$ when $f$ and $g$ are generating cofibrations
or trivial cofibrations~\cite[Corollary~4.2.5]{hovey-model}.  In this
case, $g$ will be of the form $h\otimes E$ for a map $h$ in either $I$
or $J$, and $f$ will be in either $I$ or $J$.  But then
\[
f\Box (h\otimes E) \cong (f\Box g) \otimes E
\]
so the result follows from the fact that tensoring with $E$ is a left
Quillen functor.
\end{proof}

The reader might now reasonably expect us to assert that the tensor product 
\[
X\otimes_{E} Y \mathcolon \Mod E \times (E\text{-}E'\text{-}\text{Bimod})
\xrightarrow{} \Mod E'
\]
is also a Quillen bifunctor, so that the pushout product of
cofibrations is a cofibration and so on.  However, this seems to
require $E$ to be cofibrant in $\cat{M}$, and in any case is not the
most natural thing for us to consider since we know nothing about the
structure of $FE$ as a bimodule.  

Instead, we have the following proposition.  

\begin{proposition}\label{prop-enriched-general}
Suppose $\cat{M}$ is a strongly cofibrantly generated, closed symmetric
monoidal model category, $E$ and $E'$ are monoids in $\cat{M}$, and
$A$ is an $E\text{-}E'$-bimodule that is cofibrant as a right $E'$-module.
Then the functor
\[
X\mapsto X\otimes_{E} A \mathcolon \Mod E\xrightarrow{}\Mod E'
\]
is a left Quillen functor, with right adjoint $Y\mapsto \Hom_{E'}
(A,Y)$. Similarly, if $A$ is an $E\otimes E'$-module that is fibrant as an
object of $\cat{M}$, then the functor
\[
X\mapsto \Hom_{E} (X,A) \mathcolon \Mod E\xrightarrow{}\Mod E'
\]
is a contravariant left Quillen functor, with right adjoint $Y\mapsto
\Hom_{E'} (Y,A)$.
\end{proposition}

\begin{proof}
We begin with the covariant case, and leave to the reader the check
that $\Hom_{E'} (A,-)$ is indeed right adjoint to our functor.  Note
that we use the left $E$-module structure on $A$ to make $\Hom_{E'}
(A,-)$ a right $E$-module.

We need to show that if $f$ is a cofibration or trivial cofibration in
$\Mod E$, then $f\otimes_{E}A$ is a cofibration or trivial cofibration
in $\Mod E'$.  The proof is similar in both cases, so we just work
with cofibrations.  Let $I$ be a set of generating cofibrations in
$\cat{M}$ so that $I\otimes E$ is a set of generating cofibrations in
$\Mod E$.  Then any cofibration in $\Mod E$ is a retract of a
transfinite composition of pushouts of maps of $I\otimes E$.  Since
the tensor product is a left adjoint, it preserves retracts (of
course), transfinite compositions, and pushouts.  So it suffices to
show that if $f\in I$, then
\[
(f\otimes E) \otimes_{E} A
\]
is a cofibration in $\Mod E'$.  But of course 
\[
(f\otimes E) \otimes_{E} A \cong f\otimes A,
\]
so this is ensured by the fact that $A$ is cofibrant as a right
$E'$-module and Proposition~\ref{prop-strongly}.  

The contravariant case is mostly similar.  Again, we leave to the
reader the proof of the adjointness relation, where one must take into
account the fact that the functors are contravariant.  We must then
show that if $f$ is a cofibration (resp. trivial cofibration) in $\Mod
E$, then $\Hom_{E} (f,A)$ is a fibration (resp. trivial fibration) in
$\Mod E'$, or equivalently, in $\cat{M}$.  As before, the two cases
are similar, so we only do the cofibration case.  Every cofibration in
$\Mod E$ is a retract of a transfinite composition of pushouts of maps
of $I\otimes E$, where $I$ is the set of generating cofibrations of
$\cat{M}$.  Since $\Hom_{E} (-,A)$ is a contravariant left adjoint, it
preserves retracts, converts transfinite compositions to inverse
transfinite compositions, and converts pushouts to pullbacks.  Since
retracts, inverse transfinite compositions, and pullbacks of
fibrations are fibrations, it suffices to check that
\[
\Hom_{E} (g\otimes E,A) =\Hom (g,A)
\]
 is a fibration for all $g\in I$.  But this is true because $A$ is
fibrant in $\cat{M}$.
\end{proof}

It now follows easily that the natural transformation of
Proposition~\ref{prop-general-map} descends to the homotopy category,
at least if we replace $E$ by a cofibrant approximation $QE$ before
applying $F$.  More precisely, recall that in a model category
$\cat{M}$, we can choose a functor $Q\mathcolon
\cat{M}\xrightarrow{}\cat{M}$ such that $QX$ is cofibrant for all $X$,
and a natural trivial fibration $QX\xrightarrow{}X$.

\begin{corollary}\label{cor-enriched-general}
Suppose $\cat{M}$ is a strongly cofibrantly generated, closed
symmetric monoidal model category, $E$ and $E'$ are monoids in
$\cat{M}$, and $F\mathcolon \Mod E\xrightarrow{}\Mod E'$ is a left
Quillen $\cat{M}$-functor.  Then the natural transformation $\tau$ of
Proposition~\ref{prop-general-map} has a derived natural
transformation
\[
L\tau \mathcolon QX\otimes_{E} FQE \xrightarrow{} FQX=(LF) (X) 
\]
of functors on $\ho \Mod E$.  Similarly, if $F\mathcolon \Mod
E\xrightarrow{}\Mod E'$ is a contravariant left Quillen
$\cat{M}$-functor, then the natural transformation $\tau$ of
Proposition~\ref{prop-general-map} has a derived natural
transformation
\[
L\tau \mathcolon (LF) (X)= FQX \xrightarrow{} \Hom_{E} (QX, FQE)
\]
of functors on $\ho \Mod E$.  
\end{corollary}

In applications, it is sometimes useful to note that this corollary
only requires $F$ to be left Quillen as a functor to $\Mod E'$ with
\textbf{some} model structure where the weak equivalences are the maps
of $\Mod E'$ which are weak equivalences in $\cat{M}$.  Indeed, the
model structure on $\Mod E'$ is never used in the proof of this
corollary except that $\ho \Mod E'$ is the target of the functors, and
$\ho \Mod E'$ depends only on the weak equivalences of $\Mod E'$.  

We also note that this corollary would remain true if we used $QFE$
instead of $FQE$, but that would go against one of the basic
principles of model category theory, to cofibrantly replace objects
BEFORE applying left Quillen functors.

\section{Natural transformations of Quillen functors}\label{sec-Quillen}

The object of this section is to find conditions under which the
derived natural transformation 
\[
L\tau \mathcolon LF \xrightarrow{}LG
\]
of a natural transformation 
\[
\tau \mathcolon F\xrightarrow{}G
\]
of left	Quillen functors is a natural isomorphism.  

\begin{theorem}\label{thm-derived-cofibrant}
Let $\cat{C}$ and $\cat{D}$ be model categories, $F,G\mathcolon
\cat{C}\xrightarrow{}\cat{D}$ be left Quillen functors, and $\tau
\mathcolon F\xrightarrow{}G$ be a natural transformation.  Suppose
$\cat{C}$ is cofibrantly generated so that the domains of the
generating cofibrations are cofibrant, and $\tau_{X}$ is a weak
equivalence for $X$ any domain or codomain of a generating
cofibration.  Then $L\tau\mathcolon LF\xrightarrow{}LG$ is a natural
isomorphism.  
\end{theorem}

\begin{proof}
It suffices to show that $\tau_{Y}$ is a weak equivalence for all
cofibrant $Y$.  Denote the generating cofibrations of $\cat{C}$ by
$I$.  Any cofibrant $Y$ is a retract of an object $X$ such that the
map $0\xrightarrow{}X$ is a transfinite composition of pushouts of
maps of $I$.  It then suffices to prove that $\tau_{X}$ is a weak
equivalence, which we do by transfinite induction.  The base case is
clear since $F$ and $G$ preserve the initial object.  For the
successor ordinal step, we have a pushout diagram 
\[
\begin{CD}
A @>f>> B \\
@VVV @VVV \\
X_{\alpha } @>>> X_{\alpha +1}
\end{CD}
\]
where $f$ is a map of $I$, and $\tau_{X_{\alpha}}$ is a weak
equivalence.  This gives us two pushout squares of cofibrant objects
\[
\begin{CD}
FA @>>> FB \\
@VVV @VVV \\
FX_{\alpha } @>>> FX_{\alpha +1}
\end{CD}
\]
and
\[
\begin{CD}
GA @>>> GB \\
@VVV @VVV \\
GX_{\alpha } @>>> GX_{\alpha +1}
\end{CD}
\]
in which the top horizontal map is a cofibration.  The natural
transformation $\tau$ defines a map from the top square to the bottom one,
which is a weak equivalence on every corner except possibly the
right bottom square.  The cube lemma~\cite[Lemma~5.2.6]{hovey-model}
then implies $\tau$ is a weak equivalence on the bottom right corner
as well.  

For the limit ordinal step of the induction, we have a transfinite
composition $X_{\beta}=\colim_{\alpha <\beta } X_{\alpha}$ of
cofibrant objects, where each map $X_{\alpha}\xrightarrow{}X_{\alpha
+1}$ is a cofibration.  This gives a diagram of cofibrant objects 
\[
\begin{CD}
0 @>>> \dotsb @>>> FX_{\alpha } @>>> FX_{\alpha +1} @>>> \dotsb \\
@VVV @. @V\tau_{X_{\alpha}}VV @VV\tau_{X_{\alpha +1}}V @. \\
0 @>>> \dotsb @>>> GX_{\alpha} @>>> GX_{\alpha +1} @>>> \dotsb 
\end{CD}
\]
in which each vertical arrow is a weak equivalence and each horizontal
arrow is a cofibration.  Then~\cite[Proposition~15.10.12]{hirschhorn}
implies $\phi_{X_{\beta}}$ is also a weak equivalence, completing the
induction.  This is not quite accurate, because Hirshhorn states
Proposition~15.10.12 only for ordinary sequences, but the same proof
works for transfinite sequences.  
\end{proof}

There is another version of Theorem~\ref{thm-derived-cofibrant} when
$\Mod E'$ is left proper.  Recall that a model category is
\textbf{left proper} when the pushout of a weak equivalence through a
cofibration is again a weak equivalence.  

\begin{theorem}\label{thm-derived-proper}
Let $\cat{C}$ and $\cat{D}$ be model categories, $F,G\mathcolon
\cat{C}\xrightarrow{}\cat{D}$ be left Quillen functors, and $\tau
\mathcolon F\xrightarrow{}G$ be a natural transformation.  Suppose
$\cat{C}$ is cofibrantly generated, $\cat{D}$ is left proper, and
$\tau_{X}$ is a weak equivalence for $X$ any domain or codomain of a
generating cofibration.  Then $L\tau\mathcolon LF\xrightarrow{}LG$ is
a natural isomorphism.
\end{theorem}

\begin{proof}
Use the same proof as that of Theorem~\ref{thm-derived-cofibrant},
except replace the use of the cube lemma
with~\cite[Proposition~13.5.4]{hirschhorn}, which requires that
$\cat{D}$ be left proper.  
\end{proof}

Theorem~\ref{thm-derived-cofibrant} simplifies further in the stable
situation.  Recall that a model category $\cat{C}$ is called
\textbf{stable} if it is pointed and the suspension functor is an
equivalence of $\ho \cat{C}$. This hypothesis makes $\ho \cat{C}$ a
triangulated category~\cite[Ch.7]{hovey-model}.

\begin{corollary}\label{cor-derived-cofibrant}
Let $\cat{C}$ and $\cat{D}$ be model categories, $F,G\mathcolon
\cat{C}\xrightarrow{}\cat{D}$ be left Quillen functors, and $\tau
\mathcolon F\xrightarrow{}G$ be a natural transformation.  Suppose
$\cat{C}$ is cofibrantly generated, $\cat{D}$ is stable, and
$\tau_{X}$ is a weak equivalence for $X$ any cokernel of a generating
cofibration of $\cat{C}$.  Then $L\tau\mathcolon LF\xrightarrow{}LG$
is a natural isomorphism.
\end{corollary}

\begin{proof}
The proof is the same as the proof of
Theorem~\ref{thm-derived-cofibrant} except in the successor ordinal
case.  Recall that in this case we have a pushout diagram 
\[
\begin{CD}
A @>f>> B \\
@VVV @VVV \\
X_{\alpha } @>>g_{\alpha }> X_{\alpha +1}
\end{CD}
\]
where $f$ is a generating cofibration.  Let $C$ denote the cokernel
of $f$, so that $g_{\alpha}$ is a cofibration of cofibrant objects
with cokernel $C$.  Then we have the diagram 
\[
\begin{CD}
FX_{\alpha} @>>> FX_{\alpha +1} @>>>FC @>>> \Sigma FX_{\alpha} \\
@VVV @VVV @VVV @VVV \\
GX_{\alpha} @>>> GX_{\alpha +1} @>>> GC @>>> \Sigma GX_{\alpha}
\end{CD}
\]
of exact triangles in the triangulated category $\ho \cat{D}$, where
the vertical maps are the natural transformation $L\tau$.  These
vertical maps are isomorphisms on $FX_{\alpha}$ and on $FC$, and so
must also be an isomorphism on $FX_{\alpha +1}$.  
\end{proof}

We also have another version of Corollary~\ref{cor-derived-cofibrant}
in the stable case.  For this to make sense, we recall that a
\textbf{localizing subcategory} in a triangulated category is a
full triangulated subcategory closed under retracts and arbitrary
coproducts.  The intersection of all localizing subcategories
containing a set $\cat{G}$ is written $\loc{\cat{G}}$, and is the
smallest localizing subcategory containing $\cat{G}$.  

\begin{theorem}\label{thm-derived-stable}
Let $\cat{C}$ and $\cat{D}$ be model categories, $F,G\mathcolon
\cat{C}\xrightarrow{}\cat{D}$ be left Quillen functors, and $\tau
\mathcolon F\xrightarrow{}G$ be a natural transformation.  Suppose
$\cat{C}$ and $\cat{D}$ are stable, and there is a class $\cat{G}$ of
objects of $\ho \cat{C}$ such that $\loc{\cat{G}}=\ho \cat{C}$ and
$(L\tau)_{X}$ is an isomorphism for $X\in \cat{G}$.  Then
$L\tau\mathcolon LF\xrightarrow{}LG$ is a natural isomorphism.
\end{theorem}

\begin{proof}
Because $F$ and $G$ are a left Quillen functors, $LF$ and $LG$
preserve exact triangles (see~\cite[Section~6.4]{hovey-model}),
coproducts, and suspensions.   Hence the collection of all $X$ such
that $(L\tau)_{X}$ is an isomorphism is a localizing subcategory.
Since it contains $\cat{G}$, it contains all of $\ho \cat{C}$.  
\end{proof}

\section{The Eilenberg-Watts theorem}\label{sec-general}

We can now combine the results of the last two sections to prove
homotopical versions of the Eilenberg-Watts theorem.  

\begin{theorem}\label{thm-ultimate}
Suppose $\cat{M}$ is a strongly cofibrantly generated, symmetric
monoidal model category in which the domains of the generating
cofibrations are cofibrant.  Let $E$ and $E'$ be monoids in $\cat{M}$,
and $F\mathcolon \Mod E\xrightarrow{}\Mod E'$ be a left Quillen
$\cat{M}$-functor.  Suppose that the composite
\[
A\otimes FQE\xrightarrow{}F (A\otimes QE) \xrightarrow{} F (A\otimes E)
\]
is a weak equivalence when $A$ is a domain or codomain of one of the
generating cofibrations of $\cat{M}$.  Then there is a natural
isomorphism
\[
QX\otimes_{E} FQE \xrightarrow{} FQX = (LF) (X)
\]
of functors on $\ho \Mod E$.  Similarly, if $F\mathcolon \Mod
E\xrightarrow{}\Mod E'$ is a contravariant left Quillen
$\cat{M}$-functor such that the composite 
\[
F (A\otimes E) \xrightarrow{} \Hom (A,FE) \xrightarrow{} \Hom (A, FQE)
\]
is a weak equivalence when $A$ is a domain or codomain of one of the
generating cofibrations of $\cat{M}$, then there is a natural
isomorphism
\[
(LF) (X)= FQX \xrightarrow{} \Hom_{E} (QX, FQE)
\]
of functors on $\ho \Mod E$.  
\end{theorem}

Note that in the composite 
\[
A\otimes FQE\xrightarrow{}F (A\otimes QE) \xrightarrow{} F (A\otimes E)
\]
the first map is the structure map of the $\cat{M}$-functor $F$, and
the second map is induced by the weak equivalence
$QE\xrightarrow{}E$.  A similar remark holds in the contravariant
case.  

We also note that the model structure on $\Mod E'$ that we use is
again immaterial in the proof, as long as its weak equivalences are
the maps that are weak equivalences in $\cat{M}$.  In practice, this
means that the condition that $F$ be left Quillen is not as
restrictive as it might appear at first glance.  We also point out
that in case $F$ is a \textit{strict} $\cat{M}$-functor, we should use
Theorem~\ref{thm-strict} to analyze $F$ rather than
Theorem~\ref{thm-ultimate}.

\begin{proof}
Begin with the covariant case.
Proposition~\ref{prop-enriched-general} tells us that
\[
\tau \mathcolon X\otimes_{E} FQE \xrightarrow{} FX
\]
is a natural transformation of left Quillen functors.  Now apply 
Theorem~\ref{thm-derived-cofibrant}.  The contravariant case is
similar, where we think of 
\[
\tau \mathcolon FX \xrightarrow{} \Hom_{E} (X, FQE)
\]
as a natural transformation of left Quillen functors from $\Mod E$ to
$(\Mod E')^{\textup{op}}$.  
\end{proof}

We get versions of Theorem~\ref{thm-ultimate} corresponding to each
version of Theorem~\ref{thm-derived-cofibrant}.  Corresponding to
Theorem~\ref{thm-derived-proper}, we have the following theorem.

\begin{theorem}\label{thm-left-proper}
Suppose the hypotheses of Theorem~\ref{thm-ultimate} hold true, except
that rather than assuming the domains of the generating cofibrations
of $\cat{M}$ are cofibrant, we instead assume that $\Mod E'$ is left
proper in the covariant case and right proper in the contravariant
case.  Then the conclusions of Theorem~\ref{thm-ultimate} remain true.
\end{theorem}

The reason for the right proper hypothesis in the contravariant case
is that we have left Quillen functors from $\Mod E$ to $(\Mod
E')^{\textup{op}}$, so we need $(\Mod E')^{\textup{op}}$ to be left
proper.  

Note that Theorem~\ref{thm-left-proper} will work with any left proper
model structure on $\Mod E'$ in which the weak equivalences are the
maps that are weak equivalences in $\cat{M}$.  Also, the model
structure on $\Mod E'$ that we get from the fact that $\cat{M}$ is
strongly cofibrantly generated will be left proper if $\cat{M}$ is so
and $E'$ is cofibrant in $\cat{M}$, for in this case every cofibration
in $\Mod E'$ is in particular a cofibration in $\cat{M}$.  This model
structure on $\Mod E'$ is right proper whenever $\cat{M}$ is, since
weak equivalences and fibrations are detected in $\cat{M}$.

Corresponding to Corollary~\ref{cor-derived-cofibrant}, we have the
following theorem.  

\begin{theorem}\label{thm-stable-cokernel}
Suppose $\cat{M}$ is a strongly cofibrantly generated, stable,
symmetric monoidal model category where the unit is cofibrant.  Let
$E$ and $E'$ be monoids in $\cat{M}$, and $F\mathcolon \Mod
E\xrightarrow{}\Mod E'$ be a left Quillen $\cat{M}$-functor.  Suppose
that the structure map
\[
A\otimes FE\xrightarrow{} F (A\otimes E)
\]
is a weak equivalence when $A$ is a cokernel of one of the generating
cofibrations of $\cat{M}$.  Then there is a natural isomorphism
\[
QX\otimes_{E} FE \xrightarrow{} FQX = (LF) (X)
\]
of functors on $\ho \Mod E$.  Similarly, if $F\mathcolon \Mod
E\xrightarrow{}\Mod E'$ is a contravariant left Quillen
$\cat{M}$-functor such that the structure map
\[
F (A\otimes E) \xrightarrow{} \Hom (A,FE) 
\]
is a weak equivalence when $A$ is a cokernel of one of the
generating cofibrations of $\cat{M}$, then there is a natural
isomorphism
\[
(LF) (X)= FQX \xrightarrow{} \Hom_{E} (QX, FE)
\]
of functors on $\ho \Mod E$.  
\end{theorem}

\begin{proof}
Since the unit $S$ is cofibrant in $\cat{M}$, $E$ is cofibrant in
$\Mod E$.  This is the reason we do not need $QE$ anywhere.  The main
point of the proof is that if $\cat{M}$ is stable with cofibrant unit,
so is $\Mod E'$.  Indeed, the suspension functor in $\ho \cat{M}$ or
on $\ho \Mod E'$ is the total derived functor of $S^{1}\otimes -$,
where $S^{1}$ is the suspension of the unit of $\cat{M}$ (this is
where we are assuming the unit is cofibrant).  If the suspension
functor is an equivalence on $\ho \cat{M}$, there is an object
$S^{-1}\in \ho \cat{M}$ so that tensoring with $S^{-1}$ is the inverse
of suspension. But then we can use $S^{-1}$ to define an inverse of
the suspension on $\ho \Mod E'$ as well.  Thus we can apply
Corollary~\ref{cor-derived-cofibrant}.
\end{proof}

Finally, we get the most direct analogue to the original
Eilenberg-Watts theorem as a corollary to
Theorem~\ref{thm-derived-stable}.  Here we need to assume $\cat{M}$ is
stable with cofibrant unit and $\ho \cat{M}$ is \textbf{monogenic}.
This means that the unit $S$ is a compact weak generator of the
triangulated category $\ho \cat{M}$.  Recall that an object $X$ of a
triangulated category $\cat{T}$ is called \textbf{compact} if
$\cat{T}(X, -)_{*}$ preserves coproducts.  The object $X$ is called a
\textbf{weak generator} if the functor $\cat{T} (X, -)_{*}$ is
faithful on objects, so that $Y=0$ in $\cat{T}$ if and only if
$\cat{T} (X,Y)_{*}=0$.  As we will discuss in the next section, many
of the most common stable symmetric monoidal categories are monogenic.
The relevance of the monogenic condition to
Theorem~\ref{thm-derived-stable} is that if $X$ is a compact weak
generator of a triangulated category $\cat{T}$ with all coproducts,
then $\loc{X}=\cat{T}$~\cite[Theorem~2.3.2]{hovey-axiomatic}.  

\begin{theorem}\label{thm-stable-monogenic}
Suppose $\cat{M}$ is a strongly cofibrantly generated, stable,
monogenic, symmetric monoidal model category where the unit is
cofibrant.  Let $E$ and $E'$ be monoids in $\cat{M}$, and $F\mathcolon
\Mod E\xrightarrow{}\Mod E'$ be a left Quillen $\cat{M}$-functor.
Then there is a natural isomorphism
\[
QX\otimes_{E} FE \xrightarrow{} FQX = (LF) (X)
\]
of functors on $\ho \Mod E$.  Similarly, if $F\mathcolon \Mod
E\xrightarrow{}\Mod E'$ is a contravariant left Quillen
$\cat{M}$-functor, then there is a natural isomorphism
\[
(LF) (X)= FQX \xrightarrow{} \Hom_{E} (QX, FE)
\]
of functors on $\ho \Mod E$.  
\end{theorem}

\begin{proof}
As discussed in the proof of Theorem~\ref{thm-stable-cokernel}, the
fact that $\cat{M}$ is stable with cofibrant unit $S$ means that $\Mod
E$ is stable for all monoids $E$.  We claim that $E$ is in fact a
compact weak generator of the category $\ho \Mod E$.  For this, let $U$
denote the forgetful functor from $\Mod E$ to $\cat{M}$, so that $U$
preserves and detects weak equivalences.  Since $U$ preserves weak
equivalences, its total right derived functor $RU$ is just $U$ itself;
that is, $(RU)X=UX$.  Hence $RU$ will also preserve coproducts.  Thus
\begin{gather*}
\ho \Mod E (E, \coprod X_{\alpha}) \cong  \ho \cat{M} (S, (RU) (\coprod
X_{\alpha})) \cong  \ho \cat{M} (S, \coprod UX_{\alpha}) \\
\cong \bigoplus \ho \cat{M} (S, UX_{\alpha}) \cong \bigoplus \ho \Mod
E (E, X_{\alpha}).
\end{gather*}
Thus $E$ is compact.  Similarly, if $\ho \Mod E (E, X)_{*}=0$, then
$\ho \cat{M} (S, UX)_{*}=0$, so $UX$ is weakly equivalent to $0=U
(0)$.  Since $U$ detects weak equivalences, $X$ is weakly equivalent
to $0$, so $E$ is a weak generator.  As mentioned above, we then get
that $\loc{E}=\ho \Mod E$ by the proof
of~\cite[Theorem~2.3.2]{hovey-axiomatic} (the statement of that
theorem makes some irrelevant assumptions about a tensor
product). Since $(L\tau)_{E}$ is an isomorphism,
Theorem~\ref{thm-derived-stable} completes the proof.  
\end{proof}

\section{Examples}\label{sec-examples}

In this section, we prove Theorem~\ref{thm-main} by applying our
versions of the Eilenberg-Watts theorem to the standard model
categories of symmetric spectra, chain complexes, simplicial sets, and
topological spaces.

We begin with the symmetric spectra of~\cite{hovey-shipley-smith},
based on simplicial sets.  This is a symmetric monoidal model category
(under the smash product) whose homotopy category is the standard
stable homotopy category of algebraic topology, so it is stable and
monogenic.  It is cofibrantly generated, it satisfies the monoid
axiom, and every object is small, so it is strongly cofibrantly
generated by~\ref{thm-strongly} of Schwede and Shipley.  The unit $S$
is cofibrant in symmetric spectra, making symmetric spectra easier for
us to handle than the $S$-modules
of~\cite{elmendorf-kriz-mandell-may}.  A monoid in symmetric spectra
is frequently called a \textbf{symmetric ring spectrum}, and the
homotopy category of symmetric ring spectra is equivalent to any other
homotopy category of $A_{\infty}$ ring
spectra~\cite{mandell-may-schwede-shipley}. A model category that is
enriched over symmetric spectra is frequently called a
\textbf{spectral model category}, and an enriched functor is called a
\textbf{spectral functor}.

Theorem~\ref{thm-stable-monogenic} then gives an Eilenberg-Watts
theorem for symmetric ring spectra.  

\begin{theorem}\label{thm-spectra}
Symmetric spectra are homotopically self-contained.
\end{theorem}

There are similar theorems for symmetric spectra based on topological
spaces and for the orthogonal spectra
of~\cite{mandell-may-schwede-shipley}.  The only (slight) subtlety is
that no nontrivial object of the category is small with respect to all
maps (although even that problem can be avoided by using the $\Delta
$-generated spaces of Jeff Smith, which are supposed to be a locally
presentable category), so one has to take a little care in proving
that the model categories in question are strongly cofibrantly
generated.

There is also a similar theorem for the $S$-modules
of~\cite{elmendorf-kriz-mandell-may}, even though the unit $S$ is not
cofibrant in that case.  The main issue here is we need
enough control over the unit to be sure that $\Mod E$ is stable and
that $E$ is a small weak generator of $\ho \Mod E$, for any
$S$-algebra $E$.  One cannot point directly to a theorem
in~\cite{elmendorf-kriz-mandell-may} that says this, but it does
follow from the results Chapter~III.  Proposition~III.1.3 is
especially relevant.  

We now turn to chain complexes.  Here the base symmetric monoidal
model category is the category of unbounded chain complexes of abelian
groups $\Ch{\Z}$, with the projective model
structure~\cite[Section~2.3]{hovey-model}.  This is a cofibrantly
generated model category satisfying the monoid axiom, in which every
object is small.  It is therefore strongly cofibrantly generated by
Theorem~\ref{thm-strongly}.  It is also stable and monogenic, and the
unit is cofibrant.  Thus Theorem~\ref{thm-stable-monogenic} applies.
A monoid in this category is a differential graded algebra and an
enriched functor is often called a DG-functor.

\begin{theorem}\label{thm-chain}
$\Ch{\Z}$ is homotopically self-contained. 
\end{theorem}

In fact, Theorem~\ref{thm-chain} is actually a special case
of~\cite[Section~6.4]{keller}, where Keller proves that any DG-functor
$F\mathcolon \Mod E \xrightarrow{}\Mod E'$ that commutes with direct
sums has $QX\otimes_{E}FE\cong (LF) (X)$.  That is, he does not assume
that $F$ is left Quillen.  In the special case when $E'$ is an
ordinary ring, we can recover Keller's result, and this is a
worthwhile exercise, as it illustrates that one can often use the
methods of Theorem~\ref{thm-stable-monogenic} to get tighter results
than one would at first expect.  The proof below likely works for
arbitrary DG-algebras as well.

We begin with a useful general lemma. 

\begin{lemma}\label{lem-homotopy}
Suppose $\cat{M}$ is a closed symmetric monoidal, strongly cofibrantly
generated, model category in which the unit $S$ is cofibrant, $E,E'$
are monoids in $\cat{M}$, and $F\mathcolon \Mod E\xrightarrow{}\Mod
E'$ is an $\cat{M}$-functor.  Then $F$ preserves the homotopy relation
on maps between cofibrant and fibrant objects, so preserves weak
equivalences between cofibrant and fibrant objects.  
\end{lemma}

Note that this means we could define a derived functor $DF$ for any
$\cat{M}$-functor $F$ via $(DF) (X)=F (QRX)$, where $R$ denotes
fibrant replacement.  However, this would be neither a left nor a
right derived functor in general, and does not seem to have good
properties without further assumptions on $F$.  

\begin{proof}
Let $I$ be a cylinder object in $\cat{M}$ for the unit $S$.  Then
$I\otimes X$ is a cylinder object for any cofibrant $E$-module $X$.
In particular, if $Y$ is a fibrant $E$-module, and $f,g\mathcolon
X\xrightarrow{}Y$ are homotopic, then there is a homotopy $I\otimes
X\xrightarrow{}Y$ between $f$ and $g$.  This corresponds to a map
$I\xrightarrow{}\Hom_{E} (X,Y)$ in $\cat{M}$.  Since $F$ is an
$\cat{M}$-functor, we get an induced map $I\mathcolon \Hom_{E'}
(FX,FY)$, which is a homotopy between $Ff$ and $Fg$.  
\end{proof}

\begin{theorem}\label{thm-keller}[Keller]
Suppose $R$ and $S$ are ordinary rings, and $F\mathcolon
\Ch{R}\xrightarrow{}\Ch{S}$ is a DG-functor that commutes with
arbitrary coproducts.  Then there is a natural isomorphism 
\[
L\tau \mathcolon QX\otimes_{E} FE \xrightarrow{} FQX = (LF) (X)
\]
of functors on $\cat{D} (R)$.  Similarly, if $F\mathcolon \Ch{R}
\xrightarrow{}\Ch{S}$ is a contravariant DG-functor that converts
coproducts to products,then there is a natural isomorphism
\[
(LF) (X)= FQX \xrightarrow{} \Hom_{E} (QX, FE)
\]
of functors on $\cat{D} (R)$.  
\end{theorem}

\begin{proof}
Lemma~\ref{lem-homotopy} tells us that a DG-functor preserves chain
homotopy.  Because every object is fibrant, then, any DG-functor $G$
has a left derived functor $(LG) (X)=G(QX)$, where $QX$ is a cofibrant
(DG-projective) replacement for $X$ (since weak equivalences between
DG-projective objects are chain homotopy equivalences).  This left
derived functor is automatically exact on $\cat{D} (R)$, as pointed
out to the author by Keller.  Indeed, we can use the injective model
structure on $\cat{D} (S)$~\cite[Theorem~2.3.13]{hovey-model}, in
which cofibrations are degreewise split monomorphisms.  As a functor
to this model structure, a DG-functor like $F$ automatically preserves
cofibrations, and hence $LF$ is exact.  To see this, we show that $F$
preserves degreewise split monomorphisms.  Indeed, if $f\mathcolon
X\xrightarrow{}Y$ is a degreewise split monomorphism, there is a $g\in
\Hom_{R} (Y,X)_{0}$ (which is, of course, not a cycle unless $f$ is
actually split) such that $gf$ is the identity.  Applying $F$, we see
that $Ff$ is also degreewise split.  Altogether then, $L\tau $ is a
natural transformation between exact coproduct-preserving functors on
$\cat{D} (R)$ from that is an isomorphism on $R$.  It is therefore an
isomorphism on the localizing subcategory generated by $R$, which is
$\cat{D} (R)$.
\end{proof}

We now consider simplicial sets, which are of course not stable.  An
excellent description of the model structure on the category $\SSet$
of simplicial sets can be found in~\cite{goerss-jardine}; there is
also a description in~\cite[Chapter~3]{hovey-model}.  We find that
$\SSet$ is a strongly cofibrantly generated, closed symmetric monoidal
(under the product) model category in which every object is
cofibrant. (The fact that every object is cofibrant makes the monoid
axiom automatic, hence $\SSet$ is strongly cofibrantly generated).  We
can therefore apply Theorem~\ref{thm-ultimate}.  The set $I$ of
generating cofibrations consists of the maps $\partial \Delta
[n]\xrightarrow{}\Delta [n]$.  We note that the vertex $n$ is a
simplicial deformation retract of $\Delta
[n]$~\cite[Lemma~3.4.6]{hovey-model}, though no other vertex is.

\begin{theorem}\label{thm-simplicial}
Simplicial sets are homotopically self-contaiend.
\end{theorem}

Pointed simplicial sets are also homotopically self-contained, and the
proof is very similar.  

\begin{proof}
We just prove the covariant case, as the contravariant case is
similar.  In view of Theorem~\ref{thm-ultimate}, we have to show that
\[
\phi_{A}\mathcolon A \times FE \xrightarrow{}F (A\times E)
\]
is a weak equivalence for $A=\partial \Delta [n]$ and $A=\Delta [n]$.
For $A=\Delta [n]$, we have a commutative diagram 
\[
\begin{CD}
\Delta [n] \times FE @>>> F (\Delta [n]\times E) \\
@VVV @VVV \\
* \times FE @>>> F (*\times E)
\end{CD}
\]
where the vertical maps are induced by the simplicial homotopy
equivalence $\Delta [n]\xrightarrow{}*$ that collapses $\Delta [n]$
onto the vertex $n$.  The same proof as that of
Lemma~\ref{lem-homotopy} implies that $F$ preserves simplicial
homotopy equivalences.  Thus the vertical maps are simplicial homotopy
equivalences, and the the bottom horizontal map is an isomorphism, so
the top map is a weak equivalence as required.

We prove that $\phi_{\partial \Delta [n]}$ is a weak equivalence by
induction on $n$.  The case $n=0$ is trivial, and the case $n=1$ is
straightforward since $\partial \Delta [1]$ is the coproduct of two
copies of $\Delta [0]$, and $F$ preserves coproducts.  Now suppose
that $\phi_{\partial \Delta [n]}$ is a weak equivalence.  We first
show that $\phi_{\Delta [n]/\partial \Delta [n]}$ is a weak
equivalence using the cube lemma~\cite[Lemma~5.2.6]{hovey-model}.
Indeed, we have two pushout squares 
\[
\begin{CD}
\partial \Delta [n] \times FE @>>> \Delta [n]\times FE \\
@VVV @VVV \\
* \times FE @>>> \Delta [n]/\partial \Delta [n] \times FE
\end{CD}
\]
and 
\[
\begin{CD}
F (\partial \Delta [n]\times E) @>>> F (\Delta [n]\times E) \\
@VVV @VVV \\
F (*\times E) @>>> F (\Delta [n]/\partial \Delta [n]\times E)
\end{CD}
\]
of cofibrant objects, where the top horizontal maps are cofibrations.
The map $\phi$ defines a map from the first pushout square to the
second, which is a weak equivalence at every spot except the lower
right corner.  The cube lemma tells us that it is also a weak
equivalence at the lower right corner.  

Now, there is a map $g\mathcolon \partial \Delta [n+1]
\xrightarrow{}\Delta [n]/\partial \Delta [n]$ of simplicial sets,
which is a weak equivalence.  It is easier to explain this map
geometrically.  The geometric realization of $\partial \Delta [n+1]$
is a triangulation of the $n$-sphere, and the geometric realization of
$\Delta [n]/\partial \Delta [n]$ is the usual CW description of an
$n$-sphere, with one point and one $n$-cell.  What we want to do is to
take one face of $\partial \Delta [n+1]$ and spread it out over the
whole $n$-cell, sending everything else in $\partial \Delta [n+1]$ to
the basepoint.  This is obviously a homotopy equivalence.  We can
realize it simplicially by sending each $k$-simplex of $\partial
\Delta [n+1]$ for $k\leq n$ except $123\dotsb n$ to the simplex
represented by $k$ $0$'s, and sending $123\dotsb n$ to the unique
nondegenerate $n$ simplex of $\Delta [n]/\partial \Delta [n]$.

We then get the commutative diagram below. 
\[
\begin{CD}
\partial \Delta [n+1] \times FE @>>> F (\partial \Delta [n+1]\times E) \\
@Vg\times FE VV @VVF (g\times E) V \\
\Delta [n]/\partial \Delta [n] \times FE @>>> F (\Delta [n]/\partial
\Delta [n]\times E)
\end{CD}
\]
The bottom horizontal map is a weak equivalence, as we have seen.  The
vertical maps are also weak equivalences, because the map $g$ is a
weak equivalence between cofibrant objects, and all functors involved
preserve those (in particular, $FE$ is cofibrant in $\Mod E'$ since
$E$ is cofibrant in $\Mod E$, so the product with $FE$ preserves such
weak equivalences).  Hence the top horizontal map is a weak
equivalence, completing the induction step and the proof.  
\end{proof}

We now consider the case when our base model category is topological
spaces.  Of course, we need a closed symmetric monoidal category of
topological spaces; for definiteness, we choose the compactly
generated weak Hausdorff spaces used
in~\cite{elmendorf-kriz-mandell-may}, and refer to this category as
$\Top $.  The model structure on $\Top$ is described
in~\cite[Section~2.4]{hovey-model}, and it is strongly cofibrantly
generated with generating cofibrations $S^{n-1}\xrightarrow{}D^{n}$
for all $n\geq 0$.  A $\Top$-functor is usually called a continuous
functor.

\begin{theorem}\label{thm-topological}
Topological spaces are homotopically self-contained.  
\end{theorem}

Again, we find similarly that pointed topological spaces are
homotopically self-contained as well.  

We note that of course the most obvious topological monoid is a
topological group $G$, in which case we are talking about $G$-spaces.
The model structure we are using on $G$-spaces is the one in which the
weak equivalences are equivariant maps which are underlying weak
equivalences.  We would like to be using the complete model structure,
where a map $f$ is a weak equivalence if and only the induced map
$f^{H}$ on $H$-fixed points is a weak equivalence for all subgroups
$H$ (in some family, perhaps).  Our methods would apply to this case,
except for one point.  We get the natural transformation
\[
X\times_{G} FG \xrightarrow{} FX
\]
but it is not clear that the left hand functor would ever be a left
Quillen functor.  

\begin{proof}
The proof is precisely analogous to that of
Theorem~\ref{thm-simplicial}.  Theorem~\ref{thm-ultimate} tells us
that we have to show that
\[
\phi_{A}\mathcolon A \times FE \xrightarrow{}F (A\times E)
\]
is a weak equivalence for $A=S^{n-1}$ and $A=D^{n}$.  The space
$D^{n}$ is contractible, and continuous functors preserve homotopy by
Lemma~\ref{lem-homotopy}, so we can use the same argument as in
Theorem~\ref{thm-simplicial} for $A=D^{n}$.  We can use induction on
$n$ for $A=S^{n-1}$, just as in Theorem~\ref{thm-simplicial}, and it
is even easier, as $D^{n}/S^{n-1}$ is homeomorphic to $S^{n}$.
\end{proof}

\section{Brown representability}\label{sec-Brown}

In this section, we point out that our results are relevant to Brown
representability of homology and cohomology theories.  For simplicity,
we stick to the stable case.  

Recall, then, that if $\cat{T}$ is a triangulated category, a
\textbf{homology functor} is an exact, coproduct-preserving functor
$h\mathcolon \cat{T}\xrightarrow{}\cat{A}$ to some abelian category
$\cat{A}$. We will refer to the graded version $h_{*}$ of $h$, defined
by $h_{n} (X)=h (\Sigma^{n}X)$, as the associated \textbf{homology
theory}.  When considering a homology theory, we need to consider the
isomorphisms $h_{n} (X)\cong h_{n+1} (\Sigma X)$ as part of the data.
Similarly, a \textbf{cohomology functor} is an exact, contravariant
functor $\cat{T}\xrightarrow{}\cat{A}$ that converts coproducts to
products, and we have a similar induced cohomology theory.  A
cohomology functor $h$ is \textbf{representable} if there is a natural
isomorphism
\[
h (X)\cong \cat{T} (X,Y)
\]
for some object $Y$ of $\cat{T}$, and we say that \textbf{Brown
representability for cohomology functors} holds if every cohomology
functor is representable.  This is true in considerable generality;
see~\cite[Proposition~8.4.2]{neeman-book}, for example.

Representability for homology functors is much more complicated.  Even
understanding what it means for a homology functor to be representable
is not obvious.  Since we will be working with triangulated categories
of the form $\ho \Mod E$, for $E$ a monoid in a strongly cofibrantly
generated, closed symmetric monoidal stable monogenic model category
$\cat{M}$, the natural definition for us is that a homology functor
$h$ is \textbf{representable} if there is a natural isomorphism 
\[
h (X) \cong \ho \cat{M} (S, X \otimes_{E}^{L} Y)
\]
for some left $E$-module $Y$, where $X\otimes_{E}^{L} Y$ denotes the
derived tensor product and $S$ denotes the unit of $\cat{M}$.  Note
that this is much more subtle; for example, there is no reason to
think that a morphism between representable homology theories must be
itself represeentable by a map between the representing objects.

Of course, homology functors, and natural transformations between
them, on the stable homotopy category are representable.  The same is
true for $\cat{D} (R)$ for countable rings $R$.  However, Christensen,
Keller, and Neeman proved in~\cite{christensen-keller-neeman} that
there are rings $R$ for which not every homology theory on $\cat{D}
(R)$ is representable.  Before that, Beligiannis~\cite{beligiannis}
had proved that natural transformations between representable homology
functors on $\cat{D} (R)$ need not be representable.

We can use our versions of the Eilenberg-Watts theorem to partially
salvage Brown representability for homology theories.  

\begin{definition}\label{defn-strict-model}
Suppose $\cat{M}$ is a strongly cofibrantly generated, stable,
monogenic, symmetric monoidal model category where the unit $S$ is
cofibrant.  Let $E$ be a monoid in $\cat{M}$, and $h\mathcolon \ho
\Mod E\xrightarrow{}\Mod A$ be a homology (resp. cohomology) functor,
where $A$ is an ordinary ring.  We say that $h_{*}$ has a
\textbf{strict model} if there is a monoid $E'$ in $\cat{M}$ with $\ho
\cat{M} (S,E')_{*}\cong A_{*}$ as rings, a left Quillen
$\cat{M}$-functor (resp. a contravariant left Quillen
$\cat{M}$-functor) $F\mathcolon \Mod E\xrightarrow{}\Mod E'$, and a
natural isomorphism
\[
\rho \mathcolon h(X)\cong \ho \Mod E' (E',FQX)\cong \ho
\cat{M} (S, FQX)
\]
of $A$-modules
\end{definition}

Here, then, is our version of Brown representability, which follows
immediately from Theorem~\ref{thm-stable-monogenic}.  

\begin{theorem}\label{thm-Brown}
Suppose $\cat{M}$ is a strongly cofibrantly generated, stable,
monogenic, symmetric monoidal model category where the unit $S$ is
cofibrant.  Let $E$ be a monoid in $\cat{M}$, and $h\mathcolon \ho
\Mod E\xrightarrow{}\Mod A$ be a homology \ulp resp. cohomology\urp
theory with a strict model $F\mathcolon \Mod E\xrightarrow{}\Mod E'$.
Then $h$ is representable.  More precisely, in the homology case, we
have a natural isomorphism
\[
h (X) \cong \ho \Mod E(E, QX\otimes_{E} FE), 
\]
and in the cohomology case, we have a natural isomorphism 
\[
h (X) \cong \ho \Mod E(QX, FE).  
\]
\end{theorem}

We point out that, although the hypotheses in Theorem~\ref{thm-Brown}
are much stronger than in usual forms of Brown representability, the
conclusion is also stronger.  Typically, Brown representability
theorems just say that a cohomology functor $h$ is representable by an
object of $\ho \Mod E$, but this theorem says that the representing
object is actually an $E\otimes E'$-module.  

In practice, we usually do not need to assume quite so much to get a
Brown representability theorem.  We will illustrate this in the case of
chain complexes.  The main point is that every object is fibrant in
$\Ch{\Z}$.  This means that, if $F\mathcolon
\Ch{R}\xrightarrow{}\Ch{S}$ is a DG-functor, then $F$ preserves
weak equivalences between cofibrant objects (see
Lemma~\ref{lem-homotopy}), and so has a left derived functor $(LF)
(X)=F (QX)$.  We remind the reader that we can think of suspension in
$\cat{D} (R)$ as $\Sigma X=S^{1}\otimes QX$, where $S^{1}$ is the
complex which is $\Z$ in degree $1$ and $0$ elsewhere.  

\begin{definition}\label{defn-chain-model}
Suppose $R$ and $S$ are rings, and $h_{*}\mathcolon \cat{D}
(R)\xrightarrow{}\Mod S$ is a homology (resp. cohomology) theory.  We
say that $h_{*}$ has a \textbf{chain model} if there is a
DG-functor $F\mathcolon \Ch{R}\xrightarrow{}\Ch{S}$ and a
natural isomorphism
\[
\rho \mathcolon h_{*} (X)\cong H_{*} (FQX)
\]
of $S$-modules that is compatible with the suspension.  To explain
this, we assume $h_{*}$ is a homology theory and leave the evident
modifications in the cohomology case to the reader.  To say that $\rho
$ is compatible with the suspension means that the isomorphism $h_{n}
(X)\cong h_{n+1} (\Sigma X)$ corresponds to the composite
\[
H_{n} (FQX) \cong H_{n+1} (S^{1} \otimes FQX) \xrightarrow{} H_{n+1}
(F (S^{1}\otimes QX))\cong H_{n+1} (FQ\Sigma X), 
\]
where the last isomorphism comes from the fact that $F$ preserves the
homology isomorphism 
\[
Q\Sigma X = Q (S^{1}\otimes QX) \xrightarrow{} S^{1}\otimes QX
\]
between cofibrant objects.  We remind the reader 
\end{definition}

\begin{theorem}\label{thm-Brown-chain}
Suppose $R$ and $S$ are rings and $h_{*}\mathcolon \cat{D}
(R)\xrightarrow{} (\Mod S)_{*}$ is a homology or cohomology theory
with a chain model $F\mathcolon \Ch{R}\xrightarrow{}\Ch{S}$.  Then
$h_{*}$ is representable.  More precisely, in the homology case, there
is a natural isomorphism
\[
h_{*} (X) \cong H_{*} (QX\otimes_{E} FE)
\]
and in the cohomology case there is a natural isomorphism 
\[
h^{*} (X) \cong \cat{D} (R) (QX, FE)^{*}.  
\]
\end{theorem}

Again we remind the reader that the representing object for $h_{*}$
in the above theorem is a complex of $R\text{-}S$-bimodules, and for $h^{*}$
it is a complex of $R\otimes S$-modules.  The usual Brown
representability theorems, when they apply, would just give an action
of $S$ on the representing object up to homotopy. 

\begin{proof}
We assume $h_{*}$ is a homology theory, and leave the modifications in
the cohomology case to the reader.  Since $F$ is a
DG-functor, we have a natural
transformation (by Proposition~\ref{prop-general-map})
\[
\tau \mathcolon X \otimes_{E} FE \xrightarrow{} FX.  
\]
Since both the domain and target of this natural transformation
preserve weak equivalences between cofibrant objects, there is a
derived natural transformation 
\[
L\tau \mathcolon QX\otimes_{E} FE \xrightarrow{} (LF) (X)=F (QX).  
\]
This natural transformation is an isomorphism when $X=S^{0}R$.  As in
the proof of Theorem~\ref{thm-stable-monogenic}, the localizing
subcategory generated by $S^{0}R$ is all of $\cat{D} (R)$.  If we knew
$F$ were a left Quillen functor, we could then use
Theorem~\ref{thm-derived-stable} to conclude that $L\tau$ is an
isomorphism for all $X$.  Instead, we use the fact that $h$ is a
homology theory to conclude that $LF$ is exact and preserves
coproducts and suspensions, giving the desired result.

To see that $LF$ preserves coproducts, we just note that the natural
map
\[
\coprod_{i} (LF)X_{i} \xrightarrow{} (LF) (\coprod_{i} X_{i})
\]
becomes an isomorphism on applying $H_{*}$, and is
therefore an isomorphism in $\cat{D} (S)$.  
Because $F$ is a DG-functor, there is a natural map
\[
S^{1} \otimes QFQX\xrightarrow{} F (S^{1} \otimes QX),
\]
and our hypothesis that the isomorphism 
\[
h_{*} (X)\cong H_{*} (FQX)
\]
is compatible with the suspension guarantees that this is a homology
isomorphism, so is an isomorphism in $\cat{D} (S)$.  Hence $LF$
commutes with the suspension.

To see that $LF$ preserves exact triangles, we need to recall a little
more about exact triangles in $\cat{D} (R)$.  Any such exact triangle
comes from a cofibration $f\mathcolon X\xrightarrow{}Y$ of cofibrant
objects and is isomorphic in $\cat{D} (R)$ to the sequence
\[
X \xrightarrow{f} Y \xrightarrow{} C (f) \xrightarrow{} \Sigma X
\]
where $\Cyl (f)$ is the mapping cylinder of $f$ and $C (f)$ is the
mapping cone. So $\Cyl (f)$ is the pushout of
$Y\xleftarrow{f}X\xrightarrow{i_{1}\otimes X}I\otimes X$, where $I$ is
the chain complex previously mentioned, with $\Z$ in degree $1$ and
$\Z \oplus \Z$ in degree $0$.  The map $i_{1}$ hits the ``right
endpoint'' copy of $\Z$ in degree $0$.  There is a chain homotopy
equivalence $\Cyl (f)\xrightarrow{}Y$ corresponding to ``stepping on
the cylinder''.  Then $Cf$ is the quotient of the composite 
\[
X \xrightarrow{i_{0}\otimes 1} I\otimes X\xrightarrow{}\Cyl (f),
\]
so that $(Cf)_{n}=Y_{n}\oplus X_{n-1}$, and the map
$Cf\xrightarrow{}\Sigma X$ is the quotient of $Y\xrightarrow{}C (f)$.

We will show that there is a commutative diagram 
\[
\begin{CD}
FX @>>> \Cyl (Ff) @>>> C (Ff) @>>> \Sigma FX \\
@| @VVV @VVV @VVV \\
FX @>>> F (\Cyl f) @>>> F (Cf) @>>> F (\Sigma X)
\end{CD}
\]
in which the vertical maps are homology isomorphisms.  The rightmost
vertical map comes from the fact that $F$ is a $\Ch{\Z }$-functor, and
we have already seen that this map is a homology isomorphism (for
DG-projective $X$).  The map $\Cyl (Ff)\xrightarrow{}F (\Cyl f)$
also exists because $F$ is a $\Ch{\Z}$ functor; it is the map
$FY\xrightarrow{}F (\Cyl f)$ on $FY$ and the map 
\[
I\otimes FX \xrightarrow{} F (I\otimes X) \xrightarrow{} FY
\]
on $I\otimes FX$.  Since $\Cyl (f)\xrightarrow{}Y$ is a chain homotopy
equivalence, so is $F (\Cyl f)\xrightarrow{}FY$.  Of course, $\Cyl
(Ff)\xrightarrow{}FY$ is also a chain homotopy equivalence, from which
we conclude that $\Cyl (Ff)\xrightarrow{}F (\Cyl f)$ is a homology
isomorphism.  

By taking quotients with some care, using the fact that $F$ is a
$\Ch{\Z}$-functor again, we get an induced map $C (Ff)\xrightarrow{}F
(Cf)$ making the diagram above commute on both sides.  Now, the top
row has a long exact sequence in homology, and the fact that $h_{*}$
is a homology theory means that the bottom row also has a long exact
sequence in homology.  The 5-lemma then implies that the map $C
(Ff)\xrightarrow{}F (Cf)$ is a homology isomorphism, completing the
proof that $LF$ is a coproduct-preserving triangulated functor.  
\end{proof}

We point out that there is a similar theorem to
Theorem~\ref{thm-Brown} for
morphisms between homology and cohomology functors with a strict
model. If such a morphism is induced by an $\cat{M}$-natural
transformation of the strict model, then it is representable as a
morphism between the representing objects.  There is also an analog of
Theorem~\ref{thm-Brown-chain} for morphisms.  

\providecommand{\bysame}{\leavevmode\hbox to3em{\hrulefill}\thinspace}
\providecommand{\MR}{\relax\ifhmode\unskip\space\fi MR }
\providecommand{\MRhref}[2]{%
  \href{http://www.ams.org/mathscinet-getitem?mr=#1}{#2}
}
\providecommand{\href}[2]{#2}


\end{document}